\documentclass[12pt]{amsart}

\usepackage{amssymb,enumitem,amsaddr}
\makeatletter
\@namedef{subjclassname@2020}{%
  \textup{2020} Mathematics Subject Classification}
\makeatother

\newtheorem{theorem}{Theorem}[section]

\newtheorem{lemma}[theorem]{Lemma}

\numberwithin{subcase}{case}
\theoremstyle{definition}
\newtheorem{definition}[theorem]{Definition}

\numberwithin{equation}{section}

\begin{document}
\baselineskip=17pt
\title[On Maximal Prime Gaps]{On Maximal Prime Gaps}

\author[Ch.T. Wang]{Cheng-TIng Wang}\footnote{Independent Researcher}\footnote{Present Address: 2F., No. 382, Daye Rd., Beitou Dist., Taipei City 112029, Taiwan}
\email{ssssasasaaasrhs.triangle@gmail.com}

\date{}

\begin{abstract}
In this paper, we show that the gap between a prime number $p_n$ and its consecutive prime $p_{n+1}$ is on the order of $\log^2{p_n}$; consequently, we show the existence of a prime in certain types of intervals.
\end{abstract}



\maketitle

\section{Introduction}

The distribution of prime numbers is a major area of interest in number theory; several conjectures have been proposed for this topic. An early conjecture about the maximal size of prime gaps was made by Bertrand(1845)\cite{bertrand1845}, in which he conjectured that given a positive integer $x$, there is at least one prime number between $x$ and $2x$, and his conjecture has been known as the \textit{Bertrand's Postulate}. Bertrand's conjecture was later proven by Chebyshev(1852) \cite{chebyshev1852} as a corollary of his proof that the number of primes not larger than $x$, denoted by $\pi(x)$, is on the order of $\frac{x}{\log{x}}$. In his proof, Chebyshev introduced the \textit{Chebyshev functions} $\varphi(x)$ and $\psi(x)$, which are the sum of the logarithms of all prime numbers not larger than $x$ and the logarithms of all powers of primes not larger than $x$, respectively. Later, Erd\H{o}s gave an elementary proof of Bertrand's Postulate in 1932,\cite{erdos1932}. The technique used by Erd\H{o}s in his proof is elementary and refined versions of the technique have been used by few others to give a simple proof of the existence of a prime number in intervals of the type $[kn, (k+1)n]$ where $k$ is a fixed positive inter much smaller than $n$, like the the existence of a prime number in the interval $[2n, 3n]$(Bachraoui, 2006\cite{Bachraoui2006}) and in the integer $[3n, 4n]$(Loo, 2011\cite{AndyLoo2011}) for all positive integers $n$.

Some other improvements on the existence of primes in the intervals of the type $[kn, (k+1)n]$ are as follows:
\begin{itemize}
    \item By using estimations about the Chebyshev Function $\vartheta(x)$, Nagura(1952) proved that for all $25\le n$, there's a prime number in the interval $[n, \frac{6}{5}n]$\cite{Nagura1952}
    \item Schoenfeld(1976) showed that for all $2010760\le n$, there's a prime number in the interval $(n, (1+\frac{1}{16597})n)$\cite{Schoenfeld1976}
    \item In his doctoral thesis, Dusart(1998) showed that for all positive integers $x\ge 3275$, there is a prime $p$ such that $x \le p < x(1+\frac{1}{2\log^2{x}})$.\cite{dusart1998}
\end{itemize}

Besides, Cram\'{e}r(1936) has shown that by assuming the Riemann Hypothesis, the gap between a prime $p$ and the next prime is $O(\sqrt{p}\log{p})$.\cite{cramer1936}. Later, Dudek(2014) found an explicit interval for Cram\'{e}r's result by showing that if the Riemann Hypothesis is true, then for all $x \ge 2$, there exists a prime number $p$ such that $x-\frac{4}{\pi}\sqrt{x}\log{x} < p \le x$\cite{dudek2014}.

Another type of result concerns the existence of prime numbers in intervals of the form $[x, x+x^{\theta}]$ for sufficiently large $x$ and for a real number $\theta < 1$. Hocheisel(1930) first proved this result. He showed that given a real number $\theta < 1$, then $\pi(x+x^{\theta})-\pi(x)\sim\frac{x^{\theta}}{\log{x}}$ as $n\to\infty$, where $\pi(x)$ is the number of prime numbers not greater than $x$, which implies that for sufficiently large positive integer $n$, $p_{n+1}-p_n < p_n^{\theta}$. Hocheisel also showed that one may take $\theta=\frac{29999}{30000}$.\cite{Hocheisel1930} Some other results in this direction are as follows:

\begin{itemize}
    \item Heilbronn(1933) showed that showed that one may take $\theta=\frac{249}{250}$.\cite{Heilbronn1933berDP}
    \item Chudakov(1936) showed that showed that for any $0 < \epsilon$ one may take $\theta=\frac{3}{4}+\epsilon$.\cite{Chudakov1936}
    \item Ingham(1937) showed that if $\zeta(s)$ is the Riemann zeta function, then if $\zeta(\frac{1}{2}+it)=O(t^c)$ for some real numbers $t$ and $c$, then $\frac{1+4c}{2+4c} < \theta$. Ingham also showed unconditionally that one may take $\frac{5}{8} < \theta$, which implies that for sufficiently large $x$, there is at least one prime number between $x^3$ and $(x+1)^3$\cite{ingham1937} and Dudek(2016) found an explicit $x$ for Ingham's result that $x \ge e^{e^{33.3}}$\cite{dudek2016}
    \item Huxley(1971) showed one may take $\frac{7}{12} < \theta$.\cite{Huxley1971}
    \item Baker, Harman and Pintz(2001) have shown that for all sufficiently large $x$, there is at least one prime number between $x-x^{0.525}$ and $x$ by using a sieve method.\cite{BHP2001}
\end{itemize}
To the author's knowledge, the result of Baker et. al. is the best unconditional result in prime gaps so far.

On the other hand, conjectures stronger than the result implied by assuming the Riemann Hypothesis have been proposed. Below are some such conjectures:
\begin{itemize}
    \item Legendre conjectured that for every positive integer $x$, there is at least one prime number between $x^2$ and $(x+1)^2$.\cite{legender1808}
    \item Oppermann(1877) made a conjecture slightly stronger than that of Legendre that for every positive integer $x$, there is at least one prime number between $x(x-1)$ and $x^2$, and a prime between $x^2$ and $x(x+1)$.\cite{oppermann1877}
    \item Brocard conjectured that if $p_n$ is the $n$th prime with $n\ge2$, than there are at least four primes between $p_n^2$ and $p_{n+1}^2$\cite{brocard1904}
    \item Andrica(1986) conjectured that if $p_n$ is the $n$th prime, then $\sqrt{p_{n+1}}-\sqrt{p_n} < 1$\cite{andrica1986}
\end{itemize}
Since these conjectures are beyond the result attainable by assuming the Riemann Hypothesis, they are widely considered as results deeper than the Riemann Hypothesis, and none of these stronger conjectures have been proven so far.

Some even stronger conjectures are as follows:
\begin{itemize}
    \item Cram\'{e}r(1936) conjectured that $\limsup_{n \to \infty} \frac{p_{n+1}-p_n}{\log^2{p_n}}=1$.(\cite{cramer1936})
    \item Firoozbakht(1982) conjectured that if $p_n$ is the $n$th prime, then $p_{n+1}^n < p_n^{n+1}$\cite{rivera}
\end{itemize}
It has been shown that the conjecture made by Firoozbakht implies that $g_n < \log^2{p_n}-\log{p_n}$ for all $5\le n$(\cite{sinha2010}) and $g_n < \log^2{p_n}-\log{p_n}-1$ for all $10\le n$(\cite{Kourbatov2015}), indicating that Firoozbakht conjecture is stronger than Cram\'{e}r's conjecture; however, while computational data like those done by Nicely(1999) suggest that $g_n < \log^2{p_n}$,\cite{Nicely1999}, giving evidence for Cram\'{e}r's conjecture that $\limsup_{n \to \infty} \frac{p_{n+1}-p_n}{\log^2{p_n}}=1$, a paper by Granville(1995) have casted doubts on the upper limit for prime gaps suggested by Cram\'{e}r's conjecture and Firoozbakht conjecture have (\cite{granville1995}). In Granville's paper, it has been pointed out that Cram\'{e}r's conjecture is inconsistent with Maier's theorem, and he instead suggested that $\limsup_{n \to \infty} \frac{p_{n+1}-p_n}{\log^2{p_n}}=c$ for some $c \ge \frac{2}{e^{\gamma}}$, where $\gamma=0.5772\ldots$ is the Euler-Mascheroni constant.

In this paper, we show that the gap $g_n$ between a prime number $p_n$ and a consecutive prime $p_{n+1}$ is on the order of $\log^2{p_n}$ and its implications. Unless otherwise specified, $p_n$ indicates the nth prime number, $g_n=p_{n+1}-p_n$ indicates the prime gap between the nth prime and its consecutive prime, $\log{y}$ indicates the natural logarithm of $y$, $\log^c{y}$ indicates $(\log{y})^c$, $\log_k{y}$ indicates $\log{\log_{k-1}{y}}$ with $\log_1{y}=\log{y}$.

\section{Main Result}
\begin{definition}$A_k=\frac{1}{k}\sum_{i=1}^{k}g_k=\frac{p_{k+1}-2}{k}$ where $1\le k$ is a positive integer.\end{definition}

From the definition of $A_n$, we immediately have $A_n-A_{n-1}=\frac{g_n-A_{n-1}}{n}=\frac{g_n-A_n}{n-1}$ for all positive integers $2\le n$ and $\frac{A_n-A_{n-1}}{A_{n-1}-A_{n-2}}=\frac{g_n-A_n}{g_{n-1}-A_{n-2}}$ for all $3\le n$; besides, $A_n$ is rational by definition.

Moreover, from the definition of $A_n$, we also have $\frac{g_n-A_{n-1}}{g_{n-1}-A_{n-1}}=\frac{n}{n-2}\frac{g_n-A_{n-1}}{g_{n-1}-A_{n-2}}=\frac{n}{n-2}\frac{A_n-A_{n-1}}{A_{n-1}-A_{n-2}}$. Now define $v=\frac{g_n-A_{n-1}}{g_{n-1}-A_{n-1}}$.

\begin{theorem}
We have the following:
\begin{itemize}
    \item $\log{(n+1)} < A_n$ whenever $20\le n$.
    \item $1\le A_{n}$ for all $1\le n$.
    \item $A_n < \log{p_n}$ whenever $61\le n$.
    \item $A_n\neq A_{n-1}$ for all $2\le n$.
\end{itemize}
\end{theorem}
\begin{proof}
Since $n(\log{n}+\log_2{n}-1) < p_n$ for all $3\le n$\cite{dusart1999} and $p_n < n(\log{n}+\log_2{n})$ for all $6\le n$\cite{rosser1941}, we have $\log{(n+1)}+\frac{(n+1)(\log_2{(n+1)}-1)-2}{n} < A_n$, then since $\frac{(n+1)(\log_2{(n+1)}-1)-2}{n}\le 0$ implies that $n\le 19$, for all $20\le n$, we have $\log{(n+1)} < \log{(n+1)}+\frac{(n+1)(\log_2{(n+1)}-1)-2}{n} < A_n$ and this proves that $\log{(n+1)} < A_n$ whenever $20\le n$. Also, since $\log{(n+1)} < A_n$ whenever $20\le n$, we have $1 < A_n$ for all $20\le n$. The case for $1\le n \le 19$ can be checked by direct calculation.

Following the inequalities for $n$th prime, if $\log{p_n} \le A_n$, then we have $\log{p_n}\le A_n < \frac{(n+1)}{n}(\log{n}+\log_2{n})$, which implies that
\begin{equation}\label{lem0-eq1}
\begin{aligned}
&\frac{(n+1)}{n}(\log{n}+\log_2{n})>A_n\ge\log{p_n}\\
&\ge\log{n}+\log{(\log{n}+\log_2{n}-1)}\\
&>\log{n}+\log_2{n}+\frac{\log_2{n}-1}{\log{n}+\log_2{n}-1}.
\end{aligned}
\end{equation}
But \eqref{lem0-eq1} implies that $\frac{\log_2{n}-1}{\log{n}+\log_2{n}-1} < \frac{\log{n}+\log_2{n}}{n}$, which indicates that $n\le 60$. Therefore, for all $61\le n$, we have $\log{p_n} < A_n$.

Now if $A_n=A_{n-1}$ for some $2\le n$, then by definition of $n$, we have $g_n-A_{n-1}=0$ and $\frac{p_{n+1}-2}{n}=A_n=A_{n-1}=\frac{p_n-2}{n-1}$, thus we have $p_{n+1}-2=\frac{n(p_n-2)}{n-1}$, which implies that $g_n=p_{n+1}-p_n=\frac{(p_n-2)}{n-1}$, therefore we have $g_n(n-1)+2=p_n$. However, since $p_n$ is always odd for all $2\le n$, $g_n=p_{n+1}-p_n$ is always even, thus $g_n(n-1)+2=p_n$ implies that $p_n$ is even, a contradiction. Therefore, $A_n\neq A_{n-1}$ for all $2\le n$.
\end{proof}

\begin{lemma}\label{lem0}
If $\frac{g_{n-1}^2}{A_{n-1}}\le g_n$ and $A_{n-1} < g_{n-1}$, then $g_{n-1} \le 2\frac{v-1}{v}A_{n-1}$.
\end{lemma}
\begin{proof}
Since both of $g_n$ and $g_{n-1}$ are positive integers, and since $\frac{g_{n-1}^2}{A_{n-1}}\le g_n$, there is a rational number $\frac{1}{A_{n-1}}\le\chi$ depending on $n$ such that $g_n=\chi g_{n-1}^2$; besides, since $A_{n-1} < g_{n-1}$, there is a positive rational number $1 < \tau$ depending on $n$ such that $g_{n-1}=\tau A_{n-1}$ and we have $0 < v$.

Therefore, by the definition of $v$, we have $v=\frac{g_n-A_{n-1}}{(\tau-1)A_{n-1}}=\frac{\chi\tau^2 A_{n-1}^2-A_{n-1}}{(\tau-1)A_{n-1}}$, which implies that $(\tau-1)v=\chi\tau^2A_{n-1}-1$. By solving the equation for $\tau$, we have $\tau=\frac{v\pm\sqrt{v^2-4\chi A_{n-1}(v-1)}}{2\chi A_{n-1}}$ with $0\le v^2-4\chi A_{n-1}(v-1)$.

If $v^2-4\chi A_{n-1}(v-1)=0$, then we have $\tau=\frac{v}{2\chi A_{n-1}}$, which implies that $v=2\chi\tau A_{n-1}=\frac{2g_n}{g_{n-1}}$. Following the definition of $v$, this implies that $(2g_n-g_{n-1})A_{n-1}=g_ng_{n-1}$ and we have $A_{n-1}=\frac{g_n}{2g_n-g_{n-1}}g_{n-1}=\frac{\chi\tau^2A_{n-1}^2}{2\chi\tau^2A_{n-1}^2-\tau A_{n-1}}\tau A_{n-1}$, which implies that $\tau=2-\frac{1}{\chi\tau A_{n-1}}$. But this equation indicates that $\tau=1+\sqrt{1-\frac{1}{\chi A_{n-1}}}$.

We have $1+\sqrt{1-\frac{1}{\chi A_{n-1}}}\le 2\frac{v-1}{v}$, this is because if $2\frac{v-1}{v} < 1+\sqrt{1-\frac{1}{\chi A_{n-1}}}$, then we have $\frac{1}{\chi A_{n-1}} < \frac{v^2-(v-2)^2}{v^2}=\frac{4(v-1)}{v^2}$, which implies that $v^2-4(v-1)\chi A_{n-1} < 0$, a contradiction. Therefore we must have $\tau=1+\sqrt{1-\frac{1}{\chi A_{n-1}}} \le 2\frac{v-1}{v}$ for the case $v^2-4\chi A_{n-1}(v-1)=0$.

Now assume that $0 < v^2-4\chi A_{n-1}(v-1)$; if we take the plus sign for the square root part of $\tau$, then we have

\begin{equation}\label{lem0-eq1}
\begin{aligned}
&\tau=\frac{v+\sqrt{v^2-4\chi A_{n-1}(v-1)}}{2\chi A_{n-1}}=\frac{v+\sqrt{v^2-4\frac{g_n}{g_{n-1}^2} A_{n-1}(v-1)}}{2\frac{g_n}{g_{n-1}^2} A_{n-1}}\\
&\implies \frac{2g_n\tau A_{n-1}-g_{n-1}^2v}{g_{n-1}^2}=\sqrt{v^2-4\frac{g_n}{g_{n-1}^2} A_{n-1}(v-1)}\\
&\implies \frac{g_n(\tau)^2A_{n-1}}{g_{n-1}^2}-v\tau =-(v-1)\implies (v-\frac{g_n}{g_{n-1}})\tau=v-1\\
&\implies\tau=\frac{g_{n-1}(v-1)}{g_{n-1}v-g_n}=\frac{g_{n-1}v-g_{n-1}}{g_{n-1}v-g_n}=1+\frac{g_n-g_{n-1}}{g_{n-1}v-g_n}\\
\end{aligned}
\end{equation}

We have $2\le\chi$ if $0 < v^2-4\chi A_{n-1}(v-1)$, this is because if $\chi < 2$, then we have $g_n < 2g_{n-1}^2$; on the other hand, since $\tau=\frac{v+\sqrt{v^2-4\chi A_{n-1}(v-1)}}{2\chi A_{n-1}} > \frac{v}{2\chi A_{n-1}}$, we have $g_{n-1}v < 2 g_n$, which indicates that $g_n < \frac{8g_n^2}{v^2}$ and we have $\frac{v^2}{8} < g_n=\tau g_{n-1}^2$. But this implies that $v < 2\sqrt{2}\sqrt{\tau}g_{n-1}$, which indicates that
\begin{equation}\label{lem0-eq2}
(\frac{1}{A_{n-1}}\tau^2+2\sqrt{2}\tau\sqrt{\tau}-2\sqrt{2}\sqrt{\tau})-\frac{1}{A_{n-1}} \le (\chi\tau^2+2\sqrt{2}\tau\sqrt{\tau}-2\sqrt{2}\sqrt{\tau})-\frac{1}{A_{n-1}} < 0.
\end{equation}
However, \eqref{lem0-eq2} implies that $\tau \le 1$, contradicting with the condition that $A_{n-1} < g_{n-1}$. Therefore, if $0 < v^2-4\chi A_{n-1}(v-1)$ then $2\le\chi$.

Now if $2\frac{v-1}{v} < \tau$, then \eqref{lem0-eq1} implies that $1-\frac{2}{v} < \frac{g_n-g_{n-1}}{g_{n-1}v-g_n}$. Now, by the definition of $v$, $\tau$ and $\chi$, we have $1-\frac{\tau-1}{\chi\tau^2 A_{n-1}-1}=1-\frac{2}{v}$ and $\frac{g_n-g_{n-1}}{g_{n-1}v-g_n}=\frac{(\chi\tau A_{n-1}-1)}{(\chi\tau A_{n-1}\frac{\tau}{\tau-1}-\frac{1}{\tau}-\chi)}$. Therefore, we have
\begin{equation}\label{lem0-eq3}
\begin{aligned}
&1-\frac{\tau-1}{\chi\tau^2 A_{n-1}-1}\le \frac{(\chi\tau A_{n-1}-1)}{(\chi\tau A_{n-1}\frac{\tau}{\tau-1}-\frac{1}{\tau}-\chi)}\\
&\implies \frac{1}{(\chi\tau A_{n-1}\frac{\tau}{\tau-1}-\frac{1}{\tau}-\chi)}(\frac{1+\tau\chi}{\tau})\le 2\frac{\tau-1}{\chi\tau^2 A_{n-1}-1}\\
&\implies \frac{\chi\tau^2 A_{n-1}-1}{\chi\tau^2 A_{n-1}-\frac{\tau-1}{\tau}-(\tau-1)\chi} < \frac{2\tau}{1+\tau\chi}\\
&\implies 1 < 1+\frac{-\frac{1}{\tau}+(\tau-1)\chi}{\chi\tau^2 A_{n-1}-\frac{\tau-1}{\tau}-(\tau-1)\chi} < \frac{2\tau}{1+\tau\chi}.
\end{aligned}
\end{equation}
But since $2\le\chi$, $\frac{2\tau}{1+\tau\chi} < 1$ for all positive $\tau$. Therefore \eqref{lem0-eq3} leads to a contradiction and we have $\tau \le 2\frac{v-1}{v}$ if we take the plus sign for $\tau=\frac{v\pm\sqrt{v^2-4\chi A_{n-1}(v-1)}}{2\chi A_{n-1}}$.

If we take the minus sign for $\tau=\frac{v\pm\sqrt{v^2-4\chi A_{n-1}(v-1)}}{2\chi A_{n-1}}$, then we have $\tau=\frac{v-\sqrt{D}}{2\chi A_{n-1}}=\frac{2(v-1)}{v+\sqrt{D}} < 2\frac{v-1}{v}$ with $D=v^2-4\chi A_{n-1}(v-1)$.

Since in all cases we have $\tau \le 2$, we have $g_{n-1} \le 2\frac{v-1}{v}A_{n-1}$. This finishes the proof.
\end{proof}
\begin{lemma}\label{lem'}
If $A_{n-1} < g_{n-1} < g_n$, then $\frac{g_{n-1}^2}{A_{n-1}}\le g_n$ for all $3\le n$.
\end{lemma}
\begin{proof}
First, if $g_n < \frac{g_{n-1}^2}{A_{n-1}}$ and $A_{n-1} < g_{n-1} < g_n$ for some $3\le n$, then we have $0 < v$, and there's a rational number $1 < \delta$ depending on $n$ such that $g_n=\frac{g_{n-1}^2}{\delta A_{n-1}}$; also, if $A_{n-1} < g_{n-1}$, then there's a rational number $1 < \tau$ depending on $n$ such that $\tau A_{n-1}=g_{n-1}$.

Let $v=\frac{g_n-A_{n-1}}{g_{n-1}-A_{n-1}}$, then by the definition of $v$, we have
\begin{equation}\label{lem'-eq1}
v=\frac{g_n-A_{n-1}}{g_{n-1}-A_{n-1}}=\frac{\frac{\tau^2 A_{n-1}^2}{\delta A_{n-1}}-A_{n-1}}{\tau A_{n-1}-A_{n-1}}=\frac{\tau^2-\delta}{\delta(\tau-1)}.
\end{equation}

On the other hand, from the definition of $v$, we have $g_{n-1}=g_n+(v-1)A_{n-1}$, which implies that $\frac{g_{n-1}^2}{\delta A_{n-1}}-g_{n-1}v+(v-1)A_{n-1}=0$. By solving the equality for $g_{n-1}$, we have $g_{n-1}=\frac{\delta A_{n-1}}{2}(v\pm\sqrt{v^2-\frac{4}{\delta }(v-1)})$.

If we take the positive sign for the square root part, then since $1 < \delta$, we have $\frac{1}{\delta^2} < \frac{1}{\delta}$, which implies that $(v-\frac{2}{\delta })^2 < v^2-\frac{4}{\delta }(v-1)$ and
\begin{equation}\label{lem'-eq2}
\delta(v-1)A_{n-1} < g_{n-1}=\frac{\delta A_{n-1}}{2}(v+\sqrt{v^2-\frac{4}{\delta }(v-1)}) < \delta v A_{n-1}.
\end{equation}

By the definition of $v$, \eqref{lem'-eq2} implies that $\frac{t^2-\delta}{\delta v-1} < v  < \frac{\tau^2-\delta}{\delta(v-1)-1}$. Since $v = \frac{\tau^2-\delta}{\delta(\tau-1)}$, this implies that $\tau-1+\frac{1}{\delta} < v$ and $v < \tau-\frac{1}{\delta}$.

On the other hand, since $1 < \delta$, we have
\begin{equation}\label{lem'-eq3}
v = \frac{\tau^2-\delta}{\delta(\tau-1)} < \frac{\tau^2-1}{\delta(\tau-1)}=\frac{\tau+1}{\delta}.
\end{equation}

But since $\tau-1+\frac{1}{\delta} < v$, \eqref{lem'-eq3} implies that $\tau-1+\frac{1}{\delta} < \frac{\tau+1}{\delta}$, which implies that $\delta < \frac{\tau}{\tau-1}$ and $\tau < \frac{\delta}{\delta-1}$.

By the premise that $g_{n-1} < g_n$ and the assumption that $\frac{g_{n-1}^2}{\delta A_{n-1}}=g_n$, we have
\begin{equation}\label{lem'-eq3-1}
\begin{aligned}
&\tau A_{n-1}=g_{n-1} < g_n=\frac{g_{n-1}^2}{\delta A_{n-1}}=\frac{\tau^2A_{n-1}}{\delta}.\\
\end{aligned}
\end{equation}
Since $\tau < \frac{\delta}{\delta-1}$, \eqref{lem'-eq3-1} implies that $1 < \frac{\tau}{\delta} < \frac{1}{\delta-1}$, which implies that $\delta < 2$.

On the other hand, since \eqref{lem'-eq3} implies that $\tau-1+\frac{1}{\delta} < v < \tau-\frac{1}{\delta}$, we have $2 < \delta$, a contradiction with \eqref{lem'-eq3-1}.

Since the original assumption leads to a contradiction, if $A_{n-1} < g_{n-1}$, then $\frac{g_{n-1}^2}{A_{n-1}}\le g_n$ for all $3\le n$.
\end{proof}

\begin{lemma}\label{lem1}
If $A_{n-1} < g_{n-1}$, then $g_n < g_{n-1}^2$ for all $3\le n$.
\end{lemma}
\begin{proof}
If $g_{n-1}^2\le g_n$, then there is a positive integer $1\le c$ depending on $n$ such that $cg_{n-1}^2\le g_n\le (c+1)g_{n-1}^2$; also, since $3\le n$ and $A_{n-1} < g_{n-1} < g_n$, we have $\frac{g_{n-1}^2}{A_{n-1}} < g_n$ and also $g_{n-1}\le 2\frac{v-1}{v}A_{n-1}$.

Since $v$ and $A_{n-1}$ are rational and $c$ is an integer, this implies that there is a positive rational number $0 \le s \le 1$ such that $g_n=(c+s)g_{n-1}^2$, write $\chi=c+s$; Also, since $A_{n-1} < g_{n-1}\le 2\frac{v-1}{v}A_{n-1}$, there is a rational number $1 < \tau \le 2\frac{v-1}{v}$ such that $g_{n-1}=\tau A_{n-1}$.

If $v^2-4c(v-1)A_{n-1} < 0$, then $cg_{n-1}^2-vg_{n-1}+(v-1)A_{n-1} \le 0$ leads to a contradiction and we have $g_n < g_{n-1}^2$. Now assume that $0\le v^2-4c(v-1)A_{n-1}$.

If $0\le v^2-4c(v-1)A_{n-1}$, then by solving the inequality $cg_{n-1}^2-vg_{n-1}+(v-1)A_{n-1} \le 0$ with $\Delta=v^2-4c(v-1)A_{n-1}$, since $1 \le c$ and $1\le A_{n-1}$, we have $\Delta < (v-2cA_{n-1})^2$; besides, we have $\frac{1}{2c}(v-\sqrt{\Delta})\le g_{n-1} \le \frac{1}{2c}(v+\sqrt{\Delta}) < \frac{1}{c}v-A_{n-1}$, which implies that $c(g_{n-1}+A_{n-1}) < v$.

On the other hand, since $g_n\le (c+1)g_{n-1}^2$, we have $0\le (c+1)g_{n-1}^2-vg_{n-1}+(v-1)A_{n-1}$, which indicates that we have either $\frac{v+\sqrt{\Delta_2}}{2(c+1)}\le g_{n-1}$, $g_{n-1} \le \frac{v-\sqrt{\Delta_2}}{2(c+1)}$ or $\Delta_2 < 0$ where $\Delta_2=v^2-4(c+1)(v-1)A_{n-1}$.

Now if $\frac{v+\sqrt{\Delta_2}}{2(c+1)}\le g_{n-1}$, then we have $v < 2(c+1)A_{n-1}$. On the other hand, since $0\le\Delta_2$, we have $4(c+1)\frac{v-1}{v}A_{n-1}\le v$.

However, since $g_{n-1} \le 2\frac{v-1}{v}A_{n-1}$ by Lemma \ref{lem0}, this implies that $v < 2(c+1)g_{n-1} \le 4(c+1)\frac{v-1}{v}A_{n-1}$. But this leads to a contradiction since $4(c+1)\frac{v-1}{v}A_{n-1}\le v$. Therefore we can't have $\frac{v+\sqrt{\Delta_2}}{2(c+1)}\le g_{n-1}$.

If $g_{n-1} \le \frac{v-\sqrt{\Delta_2}}{2(c+1)}$, then we have $g_{n-1} \le \frac{v-\sqrt{\Delta_2}}{2(c+1)} = (1+\epsilon)A_{n-1}$ where $\epsilon=\frac{v-\sqrt{\Delta_2}-2}{v+\sqrt{\Delta_2}}$. Since $g_{n-1}v=g_n+(v-1)A_{n-1}$, we have
\begin{equation}\label{lem1-eq10}
\begin{aligned}
&g_n+(v-1)(1-\frac{\epsilon}{1+\epsilon})g_{n-1} \le g_n+(v-1)A_{n-1}=g_{n-1}v\\
&\implies cg_{n-1}^2-(1-\frac{\epsilon}{1+\epsilon})g_{n-1} < \frac{\epsilon}{1+\epsilon} g_{n-1}v\\
&\implies c(1+\epsilon)g_{n-1}^2-g_{n-1} < \epsilon g_{n-1}v\\
&\implies c(1+\epsilon)g_{n-1} < \epsilon v+1.\\
\end{aligned}
\end{equation}
By the definition of $\epsilon$, \eqref{lem1-eq10} implies that
\begin{equation}\label{lem1-eq11}
\begin{aligned}
&c(1+\epsilon)g_{n-1} < \frac{4(c+1)(v-1)A_{n-1}}{(v+\sqrt{\Delta_2})^2}-\frac{2v}{v+\sqrt{\Delta_2}}+1\\
&\implies c(1+\epsilon)g_{n-1} < \frac{4(c+1)(v-1)A_{n-1}}{(v+\sqrt{\Delta_2})^2}.
\end{aligned}
\end{equation}

since $0\le\Delta_2$ implies that $4(c+1)(v-1)A_{n-1}\le v^2$, \eqref{lem1-eq11} implies that $\frac{4(c+1)(v-1)A_{n-1}}{(v+\sqrt{\Delta_2})^2}\le 1$. But this implies that $c(1+\epsilon)g_{n-1}\le 1$, which is a contradiction since we have $1\le c$, $1 < (1+\epsilon)$ and $2\le g_{n-1}$. Therefore, we can't have $g_{n-1} \le \frac{v-\sqrt{\Delta_2}}{2(c+1)}$ either.

Now if $\Delta_2 < 0$, then we have $4c(v-1)A_{n-1} \le v^2 < 4(c+1)(v-1)A_{n-1}$ and thus we have $v^2-4\chi(v-1)A_{n-1} < 4(c+1-\chi)(v-1)A_{n-1} < 4(v-1)A_{n-1}$. On the other hand, since $g_n=\chi g_{n-1}^2$, we have $\chi g_{n-1}^2-vg_{n-1}+(v-1)A_{n-1}=g_n-vg_{n-1}+(v-1)A_{n-1}=0$, which implies that $g_{n-1}=\frac{v\pm\sqrt{v^2-4\chi (v-1)A_{n-1}}}{2\chi}$.

Now if $g_{n-1}=\frac{v-\sqrt{v^2-4\chi (v-1)A_{n-1}}}{2\chi}$, then we have $(v-2\chi g_{n-1})^2=v^2-4\chi (v-1)A_{n-1} < 4(v-1)A_{n-1}$. Therefore, we have
\begin{equation}\label{lem1-eq12}
\begin{aligned}
&\chi g_{n-1}^2-vg_{n-1}+\frac{(v-2\chi g_{n-1})^2}{4} < g_n-vg_{n-1}+(v-1)A_{n-1}=0\\
&\implies 4(\chi^2+\chi)g_{n-1}^2-4(\chi+1)g_{n-1}v+v^2<0\\
&\implies 4\chi^2+4(1-z)\chi+(z^2-4z) < 0.\\
\end{aligned}
\end{equation}
Where $z=\frac{g_{n-1}}{v}$. Since $g_{n-1}=\frac{v-\sqrt{v^2-4\chi (v-1)A_{n-1}}}{2\chi}$, $0\le v^2-4\chi (v-1)A_{n-1}$ and $1\le\chi$, we have $z=\frac{g_{n-1}}{v} < 1$.

By solving the inequality $4\chi^2+4(1-z)\chi+(z^2-4z) < 0$ for $\chi$, we have $\chi < \frac{4\sqrt{(1-z)^2-(z^2-4z)}-4(1-z)}{16}=\frac{\sqrt{1+2z}-(1-z)}{4}$. Since $1\le\chi$, this implies that $4 \le \sqrt{1+2z}-(1-z)$, which indicates that $z^2-12z+24 \le 0$. By solving the inequality $z^2-12z+24 \le 0$ for $z$, we have $2.5 < \frac{12-\sqrt{144-24\times 4}}{2} \le z$. But this contradicts $z < 1$. Therefore, we can't have $\Delta_2 < 0$ either.

Since all cases implied by  the assumption $g_{n-1}^2\le g_n$ for some $3\le n$ leads to a contradiction, we have $g_n < g_{n-1}^2$ for all $3\le n$.
\end{proof}
\begin{theorem}\label{thm1}
$g_n < 4A_{n-1}^2$ whenever $4\le n$.
\end{theorem}
\begin{proof}
Without the loss of generality, assume that $A_n < g_n$, which implies that $A_{n-1} < A_n$ since $A_n=A_{n-1}+\frac{g_n-A_n}{n-1}$.

If $A_{n-1} < g_{n-1}$, then since $3 < n$, we have $\frac{g_{n-1}^2}{A_{n-1}}\le g_n$ by Lemma \ref{lem'}, and the results follows for all $3\le n$ by Lemma \ref{lem0} and Lemma \ref{lem1} since we have $g_{n-1}\le 2\frac{v-1}{v}A_{n-1} < 2A_{n-1}$; if $g_{n-1} < A_{n-1}$, then we have $A_{n-1}-A_{n-2} < 0$ and $g_{n-1} < A_{n-2}$, which and the fact that $1 < A_{n-2}, A_{n-1}, A_n$ and that $A_{n-2} < A_n$ imply that $g_n-A_n < \frac{A_n-A_{n-1}}{A_{n-2}-A_{n-1}}(A_{n-2}-g_{n-1})$ and $0 < \frac{n-2}{n}\frac{g_n-A_{n-1}}{A_{n-1}-g_{n-1}}=\frac{A_n-A_{n-1}}{A_{n-2}-A_{n-1}}$.

We have $\frac{A_n-A_{n-1}}{A_{n-2}-A_{n-1}} < \frac{A_{n-1}+A_{n-2}}{A_n}+\frac{A_n}{A_{n-2}} < \frac{7}{2}$. This is because if $\frac{A_{n-1}+A_{n-2}}{A_n}+\frac{A_n}{A_{n-2}} \le \frac{A_n-A_{n-1}}{A_{n-2}-A_{n-1}}$, then we have $\frac{A_{n-1}^2+A_{n-2}^2}{A_n}-(\frac{A_n}{A_{n-2}}-1)A_{n-1} \le 0$, but since $\frac{A_n}{A_{n-2}}-1$ is much smaller than $1$, $\frac{A_{n-1}^2+A_{n-2}^2}{A_n}-(\frac{A_n}{A_{n-2}}-1)A_{n-1} \le 0$ leads to a contradiction, therefore we have $\frac{A_n-A_{n-1}}{A_{n-2}-A_{n-1}} < \frac{A_{n-1}+A_{n-2}}{A_n}+\frac{A_n}{A_{n-2}}$; on the other hand, if $\frac{5}{2} \le \frac{A_{n-2}}{A_n}+\frac{A_n}{A_{n-2}}$, then we have either $2 \le \frac{A_n}{A_{n-2}}$ or $\frac{A_n}{A_{n-2}} \le \frac{1}{2}$, but both leads to a contradiction, thus we have $\frac{A_{n-2}}{A_n}+\frac{A_n}{A_{n-2}} < \frac{5}{2}$, which implies that $\frac{A_n-A_{n-1}}{A_{n-2}-A_{n-1}} < \frac{A_{n-1}+A_{n-2}}{A_n}+\frac{A_n}{A_{n-2}} < \frac{7}{2}$.

Now since $\frac{A_n-A_{n-1}}{A_{n-2}-A_{n-1}} < \frac{7}{2}$ and $A_{n-2} < A_{n-1}$, we have
\begin{equation}\label{thm1-case2-eq1}
\begin{aligned}
&\frac{n-2}{n}\frac{g_n-A_{n-1}}{A_{n-1}-g_{n-1}}=\frac{A_n-A_{n-1}}{A_{n-2}-A_{n-1}} < \frac{7}{2}\\
&\implies g_n < \frac{n}{n-2}\frac{7}{2}(A_{n-1}-g_{n-1})+A_{n-1} < (\frac{9}{2}\frac{n}{n-2})A_{n-1}\\
\end{aligned}
\end{equation}
However, if $4A_{n-1}^2 \le g_n$, then \eqref{thm1-case2-eq1} implies that $\log{n} < A_{n-1} < \frac{9}{8}\frac{n}{n-2}$, which indicates that $n \le 5$. Therefore, we have $g_n < 4A_{n-1}^2$ for all $6\le n$ and this completes the proof.
\end{proof}
\begin{theorem}\label{thm2}
For all positive integers $n$, $g_n < 4\log^2{p_n}$.
\end{theorem}
\begin{proof}
The case for $n\le 60$ can be checked by direct calculation; now assume that $61\le n$.

If $61\le n$, then since $61 \le n$, we have $g_n < 4A_{n-1}^2$ for all $4\le n$ and $A_n < \log{p_n}$ for all $61\le n$, the theorem follows.
\end{proof}

\section{Additional Results}
\begin{theorem}\label{thm3-1}If $2 \le a$ is a positive integer, then there is a prime number between $a(a-1)$ and $a^2$, and a prime number between $a^2$ and $a(a+1)$\end{theorem}
\begin{proof}
The case for $2\le a\le 680$ can be checked by calculation. Therefore, we shall discuss the case for $681\le a$

Suppose that for some $681 \le a$, there is no prime number between $a(a-1)$ and $a^2$, then we have $p_n < a(a-1)$ and $a^2 < p_{n+1}$.

Since $a = a^2-a(a-1) < p_{n+1}-p_n=g_n$, we have $a < g_n < 8\log^2{p_n} < 8\log^2{a^2}$. But this implies that $a\le 680$, which contradicts with the assumption that $681\le a$

Likewise, suppose that for some $681 \le a$, there is no prime number between $a^2$ and $a(a+1)$, then we have $p_n < a^2$ and $a(a+1) < p_{n+1}$.

Since $a = a(a+1)-a^2 < p_{n+1}-p_n=g_n$, following the same argument we have $a\le 680$, which again contradicts the assumption that $681\le a$.

Therefore, for all positive integers $681\le a$, there is at least a prime number between $a(a-1)$ and $a^2$, and at least a prime number between $a^2$ and $a(a+1)$.

Since the conjecture holds for all $2\le a\le 680$ as well by calculation, this implies that there is at least one prime number between $a(a-1)$ and $a^2$, and at least one prime number between $a^2$ and $a(a+1)$ for all positive integers $2\le a$.
\end{proof}
\section*{Acknowledgement}
The author would like to thank Ahmet Furkan Gocgen and Siddid Gosain for their valuable feedbacks on the manuscript.

\normalsize
\bibliographystyle{amsplain}
\bibliography{primegaps}
\end{document}